\documentclass[11pt]{article}
\usepackage{}
\usepackage{enumerate}
\usepackage{amsfonts}
\usepackage{amssymb}
\usepackage{cite}
\usepackage{xcolor}
\usepackage{mathrsfs}

\makeatletter
         \@addtoreset{equation}{section}\makeatother

\newtheorem{theorem}{Theorem}[section]

\newtheorem{lemma}[theorem]{Lemma}

\newtheorem{proposition}[theorem]{Proposition}

\newtheorem{definition}[theorem]{Definition}
\def\bbZ{\mathbb{Z}}

\def\bbR{\mathbb{R}}

\begin{document}
\title{Nonuniform average sampling in multiply generated shift-invariant subspaces of mixed
Lebesgue spaces}

\author{Qingyue Zhang\\
\footnotesize\it College of Science,  Tianjin University of Technology,
      Tianjin~300384, China\\
      \footnotesize\it     {e-mails: {jczhangqingyue@163.com}}\\ }

\maketitle

\textbf{Abstract.}\,\,
In this paper, we study nonuniform average sampling problem in multiply generated shift-invariant subspaces of mixed Lebesgue spaces. We discuss two types of average sampled values: average sampled values $\{\left \langle f,\psi_{a}(\cdot-x_{j},\cdot-y_{k}) \right \rangle:j,k\in\mathbb{J}\}$ generated by single averaging function and average sampled values $\left\{\left \langle f,\psi_{x_{j},y_{k}}\right \rangle:j,k\in\mathbb{J}\right\}$ generated by multiple averaging functions. Two fast reconstruction algorithms for this two types of average sampled values are provided. 


\textbf{2010 MR Subject Classification}\,\,
94A20, 94A12, 42C15, 41A58

\section{Introduction and motivation}
\ \ \ \
In 1961, Benedek firstly proposed  mixed Lebesgue spaces \cite{Benedek,Benedek2}. In the 1980s, Fernandez, and Francia, Ruiz and Torrea developed the theory of mixed Lebesgue spaces in integral operators and Calder\'on--Zygmund operators, respectively \cite{Fernandez,Rubio}. Recently, Torres and Ward, and Li, Liu and Zhang studied sampling problem in shift-invariant subspaces of mixed Lebesgue spaces \cite{Torres,Ward,LiLiu,zhangqingyue}. Mixed Lebesgue spaces generalize Lebesgue spaces. It was proposed due to considering functions that depend on independent quantities with different properties. For a function in mixed Lebesgue spaces, one can consider the integrability of each variable independently. This is distinct from traditional Lebesgue spaces. The flexibility of mixed Lebesgue spaces makes them have a crucial role to play in the study of time based partial differential equations. In this context, we study nonuniform average sampling problem in shift-invariant subspaces of mixed Lebesgue spaces.

Sampling theorem is the theoretical basis of modern pulse coded modulation communication system, and is also one of the most powerful basic tools in signal processing and image processing. In 1948, Shannon formally proposed  sampling theorem \cite{S,S1}. Shannon sampling
theorem  shows that for any $f\in L^{2}(\mathbb{R})$ with $\mathrm{supp}\hat{f}\subseteq[-T,T],$
$$f(x)=\sum_{n\in \mathbb{Z}}f\left(\frac{n}{2T}\right)\frac{\sin\pi(2Tx-n)}{\pi(2Tx-n)},$$
where the series converges uniformly on compact sets and in $L^{2}(\mathbb{R})$, and
$$\hat{f}(\xi)=\int_{\mathbb{R}}f(x)e^{-2\pi i x\xi }dx,\ \ \ \ \ \ \xi\in\mathbb{R}.$$
However, in many realistic situations, the sampling set is only a nonuniform
sampling set. For example, the transmission through the
internet from satellites only can be viewed as a nonuniform sampling problem, because there exists the loss of data packets in the transmission. In recent years, there are many results concerning nonuniform sampling problem \cite{SQ,BF1,M,Venkataramani,Christopher,Sun,Zhou}.
Uniform and nonuniform sampling problems
also have been generalized to more general shift-invariant spaces \cite{Aldroubi1,Aldroubi2,Aldroubi3,Aldroubi4,Vaidyanathan,Zhang1} of the form
$$V(\phi)=\left\{\sum_{k\in\mathbb{Z}}c(k)\phi(x-k):\{c(k):k\in\mathbb{Z}\}\in \ell^{2}(\mathbb{Z})\right\}.$$

In the classical sampling theory, the sampling values are the function values of the signal at the sampling points. In practical application, due to the precision of physical equipment and so on,
it is impossible to measure the exact value of a signal at a point. And the actual measured value is the local mean value of the signal near this point. Therefore, average sampling has attracted more and more the attentions of researchers \cite{SunZhou,Portal,Ponnaian,Kang,Atreas,Aldroubi10,Aldroubisun,Xianli,sunqiyu}.

For the sampling problem in shift-invariant subspaces of mixed Lebesgue spaces, Torres and Ward studied uniform sampling problem for band-limited functions in mixed Lebesgue spaces \cite{Torres,Ward}. Li, Liu and Zhang discussed the nonuniform sampling problem in principal shift-invariant spaces of mixed Lebesgue spaces \cite{LiLiu,zhangqingyue}. In this paper, we discuss nonuniform average sampling problem in multiply generated shift-invariant subspaces of mixed Lebesgue spaces. We discuss two types of average sampled values: average sampled values $\{\left \langle f,\psi_{a}(\cdot-x_{j},\cdot-y_{k}) \right \rangle:j,k\in\mathbb{J}\}$ generated by single averaging function and average sampled values $\left\{\left \langle f,\psi_{x_{j},y_{k}}\right \rangle:j,k\in\mathbb{J}\right\}$ generated by multiple averaging functions. Two fast reconstruction algorithms for this two types of average sampled values are provided. 

The paper is organized as follows. In the next section, we give the definitions and preliminary results needed and define multiply generated shift-invariant subspaces in mixed Lebesgue spaces $L^{p,q}\left(\bbR^{d+1}\right)$. In Section 3, we give main results of this paper. Section 4 gives some useful propositions and lemmas. In Section 5, we give proofs of main results. Finally, concluding remarks are presented in Section 6.

\section{Definitions and preliminary results}
\ \ \ \
In this section, we give some definitions and preliminary results needed in this paper. First of all, we give the definition of mixed Lebesgue spaces $L^{p,q}(\Bbb R^{d+1})$.
\begin{definition}
For $1 \leq p,q <+\infty$. $L^{p,q}=L^{p,q}(\Bbb R^{d+1})$ consists of all measurable functions $f=f(x,y)$ defined on $\Bbb R\times\Bbb R^{d}$ satisfying
$$\|f\|_{L^{p,q}}=\left[\int_{\Bbb R}\left(\int_{\Bbb R^d}|f(x,y)|^{q}dy\right)^{\frac{p}{q}}dx\right]^{\frac{1}{p}}<+\infty.$$
\end{definition}
The corresponding sequence spaces are defined by $$\ell^{p,q}=\ell^{p,q}(\mathbb{Z}^{d+1})=\left\{c: \|c\|^{p}_{\ell^{p,q}}=\sum_{k_{1} \in \Bbb Z}\left(\sum_{k_{2} \in \Bbb Z^d }
|c(k_{1},k_{2})|^{q}\right)^{\frac{p}{q}}
<+\infty\right\}.$$


In order to control the local behavior of functions, we introduce mixed Wiener amalgam spaces $W(L^{p,q})(\Bbb R^{d+1})$. 
\begin{definition}
For $1\leq p,q<\infty$, if a measurable function $f$ satisfies
$$\|f\|^{p}_{W(L^{p,q})}:=\sum_{n\in \Bbb Z}\sup_{x\in[0,1]}\left[\sum_{l\in \Bbb Z^d}\sup_{y\in [0,1]^d}|f(x+n,y+l)|^{q}\right]^{p/q}<\infty,$$
then we say that $f$ belongs to the mixed Wiener amalgam space $W(L^{p,q})=W(L^{p,q})(\Bbb R^{d+1})$.
\end{definition} 
For $1\leq p,q<\infty$, let $W_{0}\left ( L^{p,q} \right )$ denote the space of all continuous functions in $W(L^{p,q})$.

For $1\leq p<\infty,$ if a function $f$ satisfies
$$\|f\|^{p}_{W(L^{p})}:=\sum_{k\in \Bbb Z^{d+1}}\mathrm{ess\sup}_{x\in[0,1]^{d+1}} |f(x+k)|^{p}<\infty,$$
then we say that $f$ belongs to the Wiener amalgam space $W(L^{p})=W(L^{p})(\Bbb R^{d+1}).$
Obviously, $W(L^{p})\subset W(L^{p,p}).$

For $1\leq p< \infty $, let $W_{0}\left ( L^{p} \right )$ denote the space of all continuous functions in $W(L^{p})$. 
Let $B$ be a Banach space. $(B)^{(r)}$ denotes $r$ copies $B\times\cdots\times B$ of $B$.

For any $f,g\in L^{2}(\mathbb{R}^{d+1})$, define their convolution
$$(f*g)(x)=\int_{\mathbb{R}^{d+1}} f(y)g(x-y)dy.$$

The following is a preliminary result.

\begin{lemma}\label{convolution relation}
If $f\in L^{1}(\bbR^{d+1})$ and $g\in W(L^{1,1})$, then $f*g\in W(L^{1,1})$ and
\[
\|f*g\|_{W(L^{1,1})}\leq\|g\|_{W(L^{1,1})}\|f\|_{L^{1}}.
\]
\end{lemma}

\begin{proof}
Since
\begin{eqnarray*}
&&\|f*g\|_{W(L^{1,1})}=\sum_{k_{1}\in\bbZ}\sup_{x_{1}\in[0,1]}\sum_{k_{2}\in\bbZ^{d}}\sup_{x_{2}\in[0,1]^{d}}\\
&&\quad\quad\left|\int\limits_{y_{1}\in\bbR}\int\limits_{y_{2}\in\bbR^{d}}
f(y_{1},y_{2})g(x_{1}+k_{1}-y_{1},x_{2}+k_{2}-y_{2}))dy_{2}dy_{1}\right|\\
&&\quad\leq \sum_{k_{1}\in\bbZ}\sup_{x_{1}\in[0,1]}\sum_{k_{2}\in\bbZ^{d}}\sup_{x_{2}\in[0,1]^{d}}\\
&&\quad\quad\int\limits_{y_{1}\in\bbR}\int\limits_{y_{2}\in\bbR^{d}}
|f(y_{1},y_{2})||g(x_{1}+k_{1}-y_{1},x_{2}+k_{2}-y_{2}))|dy_{2}dy_{1}\\
&&\quad\leq \int\limits_{y_{1}\in\bbR}\int\limits_{y_{2}\in\bbR^{d}}
|f(y_{1},y_{2})|\sum_{k_{1}\in\bbZ}\sup_{x_{1}\in[0,1]}\sum_{k_{2}\in\bbZ^{d}}\sup_{x_{2}\in[0,1]^{d}}|g(x_{1}+k_{1}-y_{1},x_{2}+k_{2}-y_{2}))|dy_{2}dy_{1}\\
&&\quad\leq \int\limits_{y_{1}\in\bbR}\int\limits_{y_{2}\in\bbR^{d}}
|f(y_{1},y_{2})|\sum_{k_{1}\in\bbZ}\sup_{x_{1}\in[0,1]}\sum_{k_{2}\in\bbZ^{d}}\sup_{x_{2}\in[0,1]^{d}}|g(x_{1}+k_{1},x_{2}+k_{2}))|dy_{2}dy_{1}\\
&&\quad= \int\limits_{y_{1}\in\bbR}\int\limits_{y_{2}\in\bbR^{d}}|f(y_{1},y_{2})|\|g\|_{W(L^{1,1})}dy_{2}dy_{1}\\
&&\quad= \|g\|_{W(L^{1,1})}\int\limits_{y_{1}\in\bbR}\int\limits_{y_{2}\in\bbR^{d}}|f(y_{1},y_{2})|dy_{2}dy_{1}\\
&&\quad=\|g\|_{W(L^{1,1})}\|f\|_{L^{1}},
\end{eqnarray*}
the desired result in Lemma \ref{convolution relation} is following. 
\end{proof}

\subsection{Shift-invariant spaces}

\ \ \ \ For $\Phi=(\phi_{1},\phi_{2},\cdots,\phi_{r})^{T}\in W(L^{1,1})^{(r)}$, the multiply generated shift-invariant space in the mixed Legesgue spaces $L^{p,q}$ is defined by
\begin{eqnarray*}
&&V_{p,q}(\Phi)=\left\{\sum_{i=1}^{r}\sum_{k_{1}\in \Bbb Z}\sum_{k_{2}\in \Bbb Z^{d}}c_{i}(k_{1},k_{2})\phi_{i}(\cdot-k_{1},\cdot-k_{2}):\right.\\
&& \quad\quad\quad\quad\quad\quad\quad\quad\left.c_{i}=\left\{c_{i}(k_{1},k_{2}):k_{1}\in \Bbb Z,k_{2}\in\Bbb Z^{d}\right\}\in \ell^{p,q},\,1\leq i\leq r\right\}.
\end{eqnarray*}
It is easy to see that the three sum pointwisely converges almost everywhere. In fact, for any $1\leq i\leq r$, $c_{i}=\left\{c_{i}(k_{1},k_{2}):k_{1}\in \Bbb Z,k_{2}\in\Bbb Z^{d}\right\}\in\ell^{p,q}$ derives $c_{i}\in \ell^{\infty}.$ This combines $\Phi=(\phi_{1},\phi_{2},\cdots,\phi_{r})^{T}\in W(L^{1,1})^{(r)}$ gets
\begin{eqnarray*}
\sum_{i=1}^{r}\sum_{k_{1}\in \Bbb Z}\sum_{k_{2}\in \Bbb Z^{d}}\left|c_{i}(k_{1},k_{2})\phi_{i}(x-k_{1},y-k_{2})\right|&\leq& \sum_{i=1}^{r}\|c_{i}\|_\infty\sum_{k_{1}\in \Bbb Z}\sum_{k_{2}\in \Bbb Z^{d}}|\phi_{i}(x-k_{1},y-k_{2})|\\
&\leq&\sum_{i=1}^{r}\|c_{i}\|_\infty\| \phi_{i} \|_{W(L^{1,1})}<\infty\,(a.e.).
 \end{eqnarray*}

The following proposition gives that multiply generated shift-invariant spaces are well-defined in $L^{p,q}$.

\begin{proposition}\cite[Theorem 2.8]{zhangqingyue}\label{thm:stableup}
Assume that $1\leq p,q<\infty $ and $\Phi=(\phi_{1},\phi_{2},\cdots,\phi_{r})^{T}\in W(L^{1,1})^{(r)}$.
Then for any $C=(c_{1},c_{2},\cdots,c_{r})^{T}\in (\ell^{p,q})^{(r)}$, the function
\[
f=\sum_{i=1}^{r}\sum_{k_{1}\in \Bbb Z}\sum_{k_{2}\in \Bbb Z^d}c_{i}(k_{1},k_{2})\phi_{i}(\cdot-k_{1},\cdot-k_{2})
\]
 belongs to $L^{p,q}$ and there exist $D_{1}, D_{2}>0$ such that
$$
D_{1}\|f\|_{L^{p,q}}\leq\left(\sum_{i=1}^{r}\|c_{i}\|^{2}_{\ell^{p,q}}\right)^{1/2}\leq D_{2}\|f\|_{L^{p,q}}.
$$
\end{proposition}

\section{Main results}
\ \ \ \
In this section, we mainly discuss nonuniform average sampling in multiply generated shift-invariant spaces. The main results of this section are two fast reconstruction algorithms which allow to exactly reconstruct the signals $f$ in multiply generated shift-invariant subspaces from the average sampled values of $f$.

\subsection{The case of single averaging function}
\ \ \ \
In this subsection, we will give a fast reconstruction algorithm which allows to exactly reconstruct the signals $f$ in multiply generated shift-invariant subspaces from the average sampled values $\{\left \langle f,\psi_{a}(\cdot-x_{j},\cdot-y_{k}) \right \rangle:j,k\in\mathbb{J}\}$ of $f$. Before giving the main result of this subsection, we first give some definitions.

\begin{definition}
A bounded uniform partition of unity $\{\beta_{j,k}:j,k\in\mathbb{J}\}$ associated to $\{B_{\gamma}(x_{j},y_{k}):j,k\in\mathbb{J}\}$ is a set of functions satisfying
\begin{enumerate}
  \item $0\leq\beta_{j,k}\leq1, \forall\,j,k\in\mathbb{J},$
  \item $\mathrm{supp}\beta_{j,k}\subset B_{\gamma}(x_{j},y_{k}),$
  \item $\sum_{j\in\mathbb{J}}\sum_{k\in\mathbb{J}}\beta_{j,k}=1$.
 \end{enumerate}
Here $B_{\gamma}(x_{j},y_{k})$ is the open ball with center $(x_{j},y_{k})$ and radius $\gamma$. 
\end{definition}

If $f\in W_{0}(L^{1,1})$, we define
\[
A_{X,a}f=\sum_{j\in\mathbb{J}}\sum_{k\in\mathbb{J}}\left \langle f,\psi_{a}(\cdot-x_{j},\cdot-y_{k}) \right \rangle\beta_{j,k}=\sum_{j\in\mathbb{J}}\sum_{k\in\mathbb{J}}(f*\psi^{*}_{a})(x_{j},y_{k})\beta_{j,k},
\]
and define
\[
Q_{X}f=\sum_{j\in\mathbb{J}}\sum_{k\in\mathbb{J}}f(x_{j},y_{k})\beta_{j,k}
\]
for the quasi-interpolant of the sequence $c(j,k)=f(x_{j},y_{k})$.
Here $\psi_{a}(\cdot)=1/a^{d+1}\psi(\cdot/a)$ and $\psi^{*}_{a}(x)=\overline{\psi_{a}(-x)}$. Obviously, one has $A_{X,a}f=Q_{X}(f*\psi^{*}_{a})$.

In order to describe the structure of the sampling set $X$, we give the following definition.

\begin{definition}
If a set $X=\{(x_{j},y_{k}):k,j\in \mathbb{J},x_{k}\in\mathbb{R},y_{j}\in\mathbb{R}^{d}\}$ satisfies
\[
\bbR^{d+1}=\cup_{j,k}B_{\gamma}(x_{j},y_{k})\quad\mbox{for every}\,\gamma>\gamma_{0},
\]
then we say that the set $X$ is $\gamma_{0}$-dense in $\bbR^{d+1}$.
Here $B_{\gamma}(x_{j},y_{k})$ is the open ball with center $(x_{j},y_{k})$ and radius $\gamma$, and $\mathbb{J}$ is a countable index set.
\end{definition}

The following is main result of this subsection. 

\begin{theorem}\label{th:suanfa}
Assume that $\Phi=(\phi_{1},\phi_{2},\cdots,\phi_{r})^{T}\in W_{0}(L^{1,1})^{(r)}$ whose support is compact and $P$ is a bounded projection from $L^{p,q}$ onto $V_{p,q}(\Phi)$. Let $\psi\in W_{0}(L^{1,1})$ and $\int_{\bbR^{d+1}}\psi=1$. Then there are density $\gamma_{0}=\gamma_{0}(\Phi,\psi)>0$ and $a_{0}=a_{0}(\Phi,\psi)>0$ such that any $f\in V_{p,q}(\Phi)$ can be reconstructed
from its average samples $\{\left \langle f,\psi_{a}(\cdot-x_{j},\cdot-y_{k}) \right \rangle:j,k\in\mathbb{J}\}$ on any $\gamma\,(\gamma\leq\gamma_{0})$-dense set $X=\{(x_{j},y_{k}):j,k\in\mathbb{J}\}$  and for any $0<a\leq a_{0}$, by the following iterative algorithm:
\begin{eqnarray}\label{eq:iterative algorithm}
\left\{
\begin{array}{rl}f_{1}=&PA_{X,a}f \\
 f_{n+1}=&PA_{X,a}(f-f_{n})+f_{n}.\\
\end{array}\right.
\end{eqnarray}
The iterate $f_{n}$ converges to $f$  in the $L^{p,q}$ norm. 
Furthermore, the convergence is geometric, namely,
\[
\|f-f_{n}\|_{L^{p,q}}\leq M\alpha^{n}
\]
for some $\alpha=\alpha(\gamma,a,P,\Phi,\psi))<1$ and $M<\infty.$
\end{theorem}

\subsection{The case of multiple averaging functions}
\ \ \ \
In above subsection, we treat the case of single averaging function. However, in practice, we often encounter 
the case of multiple averaging functions. Thus, the average sampled values can be described by $\left\{\left \langle f,\psi_{x_{j},y_{k}}\right \rangle:j,k\in\mathbb{J}\right\}$.
For this case, we recover the functions $f$ exactly by using the following fast algorithm. Before giving the fast algorithm, we first define
\[
A_{X}f=\sum_{j\in\mathbb{J}}\sum_{k\in\mathbb{J}}\left \langle f,\psi_{x_{j},y_{k}} \right \rangle\beta_{j,k}.
\]

\begin{theorem}\label{th:suanfa-m}
Assume that $\Phi=(\phi_{1},\phi_{2},\cdots,\phi_{r})^{T}\in W_{0}(L^{1,1})^{(r)}$ whose support is compact and $P$ is a bounded projection from $L^{p,q}$ onto $V_{p,q}(\Phi)$.
Let the averaging sampling functions $\psi_{x_{j},y_{k}}\in W(L^{1,1})$ satisfy $\int_{\bbR^{d+1}}\psi_{x_{j},y_{k}}=1$ and $\int_{\bbR^{d+1}}|\psi_{x_{j},y_{k}}|\leq M$, where $M>0$ 
is independent of $(x_{j},y_{k})$. Then there exist density $\gamma_{0}=\gamma_{0}(\Phi,M)>0$ and $a_{0}=a_{0}(\Phi,M)>0$ such that if $X=\{(x_{j},y_{k}): j,k\in \mathbb{J}\}$  is
 $\gamma\,(\gamma\leq\gamma_{0})$-dense 
in $\bbR^{d+1}$, and if the average sampling functions $\psi_{x_{j},y_{k}}$ satisfy  $\textup{supp}\,\psi_{x_{j},y_{k}}\subseteq (x_{j},y_{k})+[-a,a]^{d+1}$ for some $0<a\leq a_{0}$, then any $f\in V_{p,q}(\Phi)$ 
can be recovered from its average samples $\left\{\left \langle f,\psi_{x_{j},y_{k}}\right \rangle:j,k\in\mathbb{J}\right\}$ by the following iterative algorithm:
\begin{eqnarray}\label{eq:iterative algorithm-m}
\left\{
\begin{array}{rl}f_{1}=&PA_{X}f \\
 f_{n+1}=&PA_{X}(f-f_{n})+f_{n}.\\
\end{array}\right.
\end{eqnarray}
In this case, the iterate $f_{n}$ converges to $f$ in the $L^{p,q}$-norm. Moreover, the convergence is geometric, that is,
\[
\|f-f_{n}\|_{L^{p,q}}\leq C\alpha^{n}
\]
for some $\alpha=\alpha(\gamma,a,P,\Phi,M)<1$ and $C<\infty.$
\end{theorem}

\section{Useful propositions and lemmas}

\ \ \ \
In this section, we introduce some useful propositions and lemmas.

Let $f$ be a continuous function. We define the oscillation (or modulus of continuity) of $f$
by $\hbox{osc}_{\delta}(f)(x_{1},x_{2})=\sup_{|y_{1}|\leq\delta,|y_{2}|\leq\delta}|f(x_{1}+y_{1},x_{2}+y_{2})-f(x_{1},x_{2})|$. 

The following tow propositions are needed in the proof of two lemmas in this section.

\begin{proposition}\cite[Lemma 3.4]{zhangqingyue}\label{pro:in Wiener space}
If $\phi\in W_{0}(L^{1,1})$ whose support is compact, then there exists $\delta_{0}>0$ such that
 $\hbox{osc}_{\delta}(\phi)\in W_{0}(L^{1,1})$ for any $\delta\leq\delta_{0}$.
\end{proposition}

\begin{proposition}\cite[Lemma 3.5]{zhangqingyue}\label{pro:oscillation}
If $\phi\in W_{0}(L^{1,1})$ whose support is compact, then $$\lim_{\delta\rightarrow0}\|\hbox{osc}_{\delta}(\phi)\|_{W(L^{1,1})}=0.$$
\end{proposition}

To prove our main results, we need the following Lemma.

\begin{lemma}\label{lem:average function}
Let $\psi\in L^{1}(\bbR^{d+1})$ satisfying $\int_{\bbR^{d+1}}\psi(x)dx=1$ and $\psi_{a}=(1/a^{d+1})\psi(\cdot/a)$. Then 
for every $\phi\in W_{0}(L^{1,1})$ whose support is compact,
\[
\|\phi-\phi*\psi^{*}_{a}\|_{W(L^{1,1})}\rightarrow0\quad \mbox{as}\quad a\rightarrow0^{+}.
\]
Here $a$ is any positive real number.
\end{lemma}

\begin{proof}
Since $\int_{\bbR^{d+1}}\psi(x)dx=1$ and $\psi^{*}_{a}(x)=\overline{\psi_{a}(-x)}$, one has
\begin{eqnarray*}
\phi-\phi*\psi^{*}_{a}=\int_{\bbR^{d+1}}(\phi(x)-\phi(x+t))\overline{\psi_{a}(t)}dt.
\end{eqnarray*}
By Proposition \ref{pro:in Wiener space}, there exists $\delta_{0}>0$ such that
 $\hbox{osc}_{\delta}(\phi)\in W_{0}(L^{1,1})$ for any $\delta\leq\delta_{0}$.
Thus
\begin{eqnarray*}
&&\|\phi-\phi*\psi^{*}_{a}\|_{W(L^{1,1})}\quad\quad\quad\quad\quad\quad\quad\quad\quad\quad\quad\quad\quad\\
&&\quad=\sum_{k_{1}\in\bbZ}\sup_{x_{1}\in[0,1]}\sum_{k_{2}\in\bbZ^{d}}\sup_{x_{2}\in[0,1]^{d}}\\
&&\quad\quad\left|\int_{\bbR}\int_{\bbR^{d}}(\phi(x_{1}+k_{1},x_{2}+k_{2})-\phi(x_{1}+k_{1}+t_{1},x_{2}+k_{2}+t_{2}))
\overline{\psi_{a}(t_{1},t_{2})}dt_{2}dt_{1}\right|\\
&&\quad\leq\sum_{k_{1}\in\bbZ}\sup_{x_{1}\in[0,1]}\sum_{k_{2}\in\bbZ^{d}}\sup_{x_{2}\in[0,1]^{d}}\\
&&\quad\quad\int_{\bbR}\int_{\bbR^{d}}\left|\phi(x_{1}+k_{1},x_{2}+k_{2})-\phi(x_{1}+k_{1}+t_{1},x_{2}+k_{2}+t_{2})\right|
\left|\psi_{a}(t_{1},t_{2})\right|dt_{2}dt_{1}\\
&&\quad\leq\sum_{k_{1}\in\bbZ}\sup_{x_{1}\in[0,1]}\sum_{k_{2}\in\bbZ^{d}}\sup_{x_{2}\in[0,1]^{d}}\\
&&\quad\quad\int\limits_{|t_{1}|\leq\delta_{0}}\int\limits_{|t_{2}|\leq\delta_{0}}\left|\phi(x_{1}+k_{1},x_{2}+k_{2})-\phi(x_{1}+k_{1}+t_{1},x_{2}+k_{2}+t_{2})\right|
\left|\psi_{a}(t_{1},t_{2})\right|dt_{2}dt_{1}\\
&&\quad\quad+\sum_{k_{1}\in\bbZ}\sup_{x_{1}\in[0,1]}\sum_{k_{2}\in\bbZ^{d}}\sup_{x_{2}\in[0,1]^{d}}\\
&&\quad\quad\int\limits_{|t_{1}|>\delta_{0}}\int\limits_{|t_{2}|\leq\delta_{0}}\left|\phi(x_{1}+k_{1},x_{2}+k_{2})-\phi(x_{1}+k_{1}+t_{1},x_{2}+k_{2}+t_{2})\right|
\left|\psi_{a}(t_{1},t_{2})\right|dt_{2}dt_{1}\\
&&\quad\quad+\sum_{k_{1}\in\bbZ}\sup_{x_{1}\in[0,1]}\sum_{k_{2}\in\bbZ^{d}}\sup_{x_{2}\in[0,1]^{d}}\\
&&\quad\quad\int\limits_{\bbR}\int\limits_{|t_{2}|>\delta_{0}}\left|\phi(x_{1}+k_{1},x_{2}+k_{2})-\phi(x_{1}+k_{1}+t_{1},x_{2}+k_{2}+t_{2})\right|
\left|\psi_{a}(t_{1},t_{2})\right|dt_{2}dt_{1}\\
&&\quad=I_{1}+I_{2}+I_{3},
\end{eqnarray*}
where
\begin{eqnarray*}
&&I_{1}=\sum_{k_{1}\in\bbZ}\sup_{x_{1}\in[0,1]}\sum_{k_{2}\in\bbZ^{d}}\sup_{x_{2}\in[0,1]^{d}}\\
&&\quad\quad\int\limits_{|t_{1}|\leq\delta_{0}}\int\limits_{|t_{2}|\leq\delta_{0}}\left|\phi(x_{1}+k_{1},x_{2}+k_{2})-\phi(x_{1}+k_{1}+t_{1},x_{2}+k_{2}+t_{2})\right|
\left|\psi_{a}(t_{1},t_{2})\right|dt_{2}dt_{1},
\end{eqnarray*}
\begin{eqnarray*}
&&I_{2}=\sum_{k_{1}\in\bbZ}\sup_{x_{1}\in[0,1]}\sum_{k_{2}\in\bbZ^{d}}\sup_{x_{2}\in[0,1]^{d}}\\
&&\quad\quad\int\limits_{|t_{1}|>\delta_{0}}\int\limits_{|t_{2}|\leq\delta_{0}}\left|\phi(x_{1}+k_{1},x_{2}+k_{2})-\phi(x_{1}+k_{1}+t_{1},x_{2}+k_{2}+t_{2})\right|
\left|\psi_{a}(t_{1},t_{2})\right|dt_{2}dt_{1}
\end{eqnarray*}
and
\begin{eqnarray*}
&&I_{3}=\sum_{k_{1}\in\bbZ}\sup_{x_{1}\in[0,1]}\sum_{k_{2}\in\bbZ^{d}}\sup_{x_{2}\in[0,1]^{d}}\\
&&\quad\quad\int\limits_{\bbR}\int\limits_{|t_{2}|>\delta_{0}}\left|\phi(x_{1}+k_{1},x_{2}+k_{2})-\phi(x_{1}+k_{1}+t_{1},x_{2}+k_{2}+t_{2})\right|
\left|\psi_{a}(t_{1},t_{2})\right|dt_{2}dt_{1}.
\end{eqnarray*}

First of all, we treat $I_{1}$: let $|t|=\max\{|t_{1}|,|t_{2}|\}$. Then
\begin{eqnarray*}
I_{1}&\leq& \sum_{k_{1}\in\bbZ}\sup_{x_{1}\in[0,1]}\sum_{k_{2}\in\bbZ^{d}}\sup_{x_{2}\in[0,1]^{d}}\int\limits_{|t_{1}|\leq\delta_{0}}\int\limits_{|t_{2}|\leq\delta_{0}}
\hbox{osc}_{|t|}(\phi)(x_{1}+k_{1},x_{2}+k_{2}))\left|\psi_{a}(t_{1},t_{2})\right|dt_{2}dt_{1}\\
&\leq& \int\limits_{|t_{1}|\leq\delta_{0}}\int\limits_{|t_{2}|\leq\delta_{0}}
\sum_{k_{1}\in\bbZ}\sup_{x_{1}\in[0,1]}\sum_{k_{2}\in\bbZ^{d}}\sup_{x_{2}\in[0,1]^{d}}\hbox{osc}_{|t|}(\phi)(x_{1}+k_{1},x_{2}+k_{2}))\left|\psi_{a}(t_{1},t_{2})\right|dt_{2}dt_{1}\\
&=&\int\limits_{|t_{1}|\leq\delta_{0}}\int\limits_{|t_{2}|\leq\delta_{0}}\|\hbox{osc}_{|t|}(\phi)\|_{W(L^{1,1})}\left|\psi_{a}(t_{1},t_{2})\right|dt_{2}dt_{1}.
\end{eqnarray*}
By proposition \ref{pro:oscillation}, for any $\epsilon>0$, there exists $\delta_{1}>0\,(\delta_{1}<\delta_{0})$ such that 
\[
\|\hbox{osc}_{\delta}(\phi)\|_{W(L^{1,1})}<\epsilon, \quad \mbox{for any}\,\delta\leq\delta_{1}.
\]
Write 
\begin{eqnarray*}
&&\int\limits_{|t_{1}|\leq\delta_{0}}\int\limits_{|t_{2}|\leq\delta_{0}}\|\hbox{osc}_{|t|}(\phi)\|_{W(L^{1,1})}\left|\psi_{a}(t_{1},t_{2})\right|dt_{2}dt_{1}\\
&&\quad=\int\limits_{|t_{1}|\leq\delta_{1}}\int\limits_{|t_{2}|\leq\delta_{1}}\|\hbox{osc}_{|t|}(\phi)\|_{W(L^{1,1})}\left|\psi_{a}(t_{1},t_{2})\right|dt_{2}dt_{1}\\
&&\quad\quad+\int\limits_{\delta_{1}<|t_{1}|\leq\delta_{0}}\int\limits_{|t_{2}|\leq\delta_{0}}\|\hbox{osc}_{|t|}(\phi)\|_{W(L^{1,1})}\left|\psi_{a}(t_{1},t_{2})\right|dt_{2}dt_{1}\\
&&\quad \quad\quad +\int\limits_{|t_{1}|\leq\delta_{1}}\int\limits_{\delta_{1}<|t_{2}|\leq\delta_{0}}\|\hbox{osc}_{|t|}(\phi)\|_{W(L^{1,1})}\left|\psi_{a}(t_{1},t_{2})\right|dt_{2}dt_{1}\\
&&\quad=I_{4}+I_{5}+I_{6}.
\end{eqnarray*}
Then 
\begin{eqnarray*}
I_{4}&\leq& \int\limits_{|t_{1}|\leq\delta_{1}}\int\limits_{|t_{2}|\leq\delta_{1}}\|\hbox{osc}_{\delta_{1}}(\phi)\|_{W(L^{1,1})}\left|\psi_{a}(t_{1},t_{2})\right|dt_{2}dt_{1}\\
&\leq& \epsilon\int\limits_{|t_{1}|\leq\delta_{1}}\int\limits_{|t_{2}|\leq\delta_{1}}\left|\psi_{a}(t_{1},t_{2})\right|dt_{2}dt_{1}\leq\epsilon\|\psi\|_{L^{1}},
\end{eqnarray*}
\begin{eqnarray*}
I_{5}&\leq&\int\limits_{\delta_{1}<|t_{1}|\leq\delta_{0}}\int\limits_{|t_{2}|\leq\delta_{0}}\|\hbox{osc}_{\delta_{0}}(\phi)\|_{W(L^{1,1})}\left|\psi_{a}(t_{1},t_{2})\right|dt_{2}dt_{1}\\
&\leq&\|\hbox{osc}_{\delta_{0}}(\phi)\|_{W(L^{1,1})}\int\limits_{\delta_{1}/a<|s_{1}|}\int\limits_{s_{2}\in\bbR^{d}}\left|\psi(s_{1},s_{2})\right|ds_{2}ds_{1}\\
&\rightarrow& 0\quad \mbox{as}\,a\rightarrow0^{+}
\end{eqnarray*}
and
\begin{eqnarray*}
I_{6}&\leq&\int\limits_{|t_{1}|\leq\delta_{1}}\int\limits_{\delta_{1}<|t_{2}|\leq\delta_{0}}\|\hbox{osc}_{|t|}(\phi)\|_{W(L^{1,1})}\left|\psi_{a}(t_{1},t_{2})\right|dt_{2}dt_{1}\\
&\leq&\|\hbox{osc}_{\delta_{0}}(\phi)\|_{W(L^{1,1})}\int\limits_{s_{1}\in\bbR}\int\limits_{\delta_{1}/a<|s_{2}|}\left|\psi(s_{1},s_{2})\right|ds_{2}ds_{1}\\
&\rightarrow& 0\quad \mbox{as}\,a\rightarrow0^{+}.
\end{eqnarray*}

Next, we treat $I_{2}$: 
\begin{eqnarray*}
&&I_{2}\leq\int\limits_{|t_{1}|>\delta_{0}}\int\limits_{|t_{2}|\leq\delta_{0}}\sum_{k_{1}\in\bbZ}\sup_{x_{1}\in[0,1]}\sum_{k_{2}\in\bbZ^{d}}\sup_{x_{2}\in[0,1]^{d}}\left|\phi(x_{1}+k_{1},x_{2}+k_{2})\right|
\left|\psi_{a}(t_{1},t_{2})\right|dt_{2}dt_{1}\\
&&\quad\quad+\int\limits_{|t_{1}|>\delta_{0}}\int\limits_{|t_{2}|\leq\delta_{0}}\\
&&\quad\quad\quad\sum_{k_{1}\in\bbZ}\sup_{x_{1}\in[0,1]}\sum_{k_{2}\in\bbZ^{d}}\sup_{x_{2}\in[0,1]^{d}}\left|\phi(x_{1}+k_{1}+t_{1},x_{2}+k_{2}+t_{2})\right|
\left|\psi_{a}(t_{1},t_{2})\right|dt_{2}dt_{1}\\
&&\quad\leq2\int\limits_{|t_{1}|>\delta_{0}}\int\limits_{|t_{2}|\leq\delta_{0}}\|\phi\|_{W(L^{1,1})}\left|\psi_{a}(t_{1},t_{2})\right|dt_{2}dt_{1}\\
&&\quad\leq2\|\phi\|_{W(L^{1,1})}\int\limits_{|t_{1}|>\delta_{0}}\int\limits_{|t_{2}|\leq\delta_{0}}\left|\psi_{a}(t_{1},t_{2})\right|dt_{2}dt_{1}\\
&&\quad\leq2\|\phi\|_{W(L^{1,1})}\int\limits_{|s_{1}|>\delta_{0}/a}\int\limits_{t_{2}\in\bbR^{d}}\left|\psi(s_{1},s_{2})\right|ds_{2}ds_{1}\\
&&\quad\rightarrow 0\quad \mbox{as}\,a\rightarrow0^{+}.
\end{eqnarray*}
Similarly, we can prove $I_{3}\rightarrow 0\quad \mbox{as}\,a\rightarrow0^{+}$. This completes the proof of Lemma \ref{lem:average function}.
\end{proof}

The following lemma is a generalization of \cite[Lemma 4.1]{Aldroubisun}.

\begin{lemma}\label{Qoperator}
Let $X$ be any sampling set which is $\gamma$-dense in $\bbR^{d+1}$, let $\{\beta_{j,k}:j,k\in\mathbb{J}\}$ be a bounded uniform partition of unity associated with $X$, and let $\phi\in W_{0}(L^{1,1})$ whose support is compact.
 Then there exist 
constants $C$ and $\gamma_{0}$ such that for any $f=\sum_{k\in\bbZ^{d+1}}c_{k}\phi(\cdot-k)$ and $\gamma\leq\gamma_{0}$, one has 
\[
\|Q_{X}f\|_{L^{p,q}}\leq C\|c\|_{\ell^{p,q}}\|\phi\|_{W(L^{1,1})} \quad \mbox{for any}\, c=\{c_{k}:k\in\bbZ^{d+1}\}\in\ell^{p,q}.
\]
\end{lemma}

To prove the Lemma \ref{Qoperator}, we introduce the following proposition.

\begin{proposition}\cite[Theorem 3.1]{LiLiu}\label{pro:stableup2}
Assume that $1\leq p,q<\infty $ and $\phi\in W(L^{1,1})$.
Then for any $c\in \ell^{p,q}$, the function $f=\sum_{k_{1}\in \Bbb Z}\sum_{k_{2}\in \Bbb Z^d}c(k_{1},k_{2})\phi(\cdot-k_{1},\cdot-k_{2})$
 belongs to $L^{p,q}$ and
$$
\|f\|_{L^{p,q}}\leq\|c\|_{\ell^{p,q}}\left \| \phi \right \|_{W(L^{1,1})}.
$$
\end{proposition}

\begin{proof}[\textbf{Proof of Lemma \ref{Qoperator}}]
By Proposition \ref{pro:stableup2} and Proposition \ref{pro:in Wiener space}, one has that there exists $\gamma_{0}>0$ such that for any $\gamma\leq\gamma_{0}$
\begin{eqnarray*}
\|f-Q_{X}f\|_{L^{p,q}}\leq\|\hbox{osc}_{\gamma}(f)\|_{L^{p,q}}\leq\|c\|_{\ell^{p,q}}\|\hbox{osc}_{\gamma}(\phi)\|_{W(L^{1,1})}.
\end{eqnarray*}
Using Proposition \ref{pro:stableup2} and the proof of \cite[Lemma 3.4]{zhangqingyue}, one obtains that there exists $C'$ such that
\begin{eqnarray*}
\|Q_{X}f\|_{L^{p,q}}&=&\|f-Q_{X}f+f\|_{L^{p,q}}\leq\|f-Q_{X}f\|_{L^{p,q}}+\|f\|_{L^{p,q}}\\
&\leq&\|c\|_{\ell^{p,q}}\|\hbox{osc}_{\gamma}(\phi)\|_{W(L^{1,1})}+\|c\|_{\ell^{p,q}}\left \| \phi \right \|_{W(L^{1,1})}\\
&\leq&\|c\|_{\ell^{p,q}}C'\|\phi\|_{W(L^{1,1})}+\|c\|_{\ell^{p,q}}\left \| \phi \right \|_{W(L^{1,1})}\\
&\leq&C\|c\|_{\ell^{p,q}}\left \| \phi \right \|_{W(L^{1,1})},
\end{eqnarray*}
where $C=1+C'$.
\end{proof}

The following lemma plays an important role in the proof Theorem \ref{th:suanfa}.

\begin{lemma}\label{lem:co}
Let $\Phi=(\phi_{1},\phi_{2},\cdots,\phi_{r})^{T}\in W_{0}(L^{1,1})^{(r)}$ whose support is compact and $P$ be a bounded projection from $L^{p,q}(\bbR^{d+1})$ onto $V_{p,q}(\Phi)$. Then there exist $\gamma_{0}> 0$ and $a_{0}> 0$ such that for $\gamma$-dense set $X$ with $\gamma\leq\gamma_{0}$ and for every positive real number $a\leq a_{0}$, the operator $I-PA_{X,a}$ is a contraction operator on $V_{p,q}(\Phi)$.
\end{lemma}

\begin{proof}
Putting $f=\sum_{i=1}^{r}\sum_{k_{1}\in \Bbb Z}\sum_{k_{2}\in \Bbb Z^{d}}c_{i}(k_{1},k_{2})\phi_{i}(\cdot-k_{1},\cdot-k_{2})\in V_{p,q}(\Phi)$. Then one has
\begin{eqnarray*}
\|f-PA_{X,a}f\|_{L^{p,q}}&=&\|f-PQ_{X}f+PQ_{X}f-PA_{X,a}f\|_{L^{p,q}}\\
&\leq&\|f-PQ_{X}f\|_{L^{p,q}}+\|PQ_{X}f-PA_{X,a}f\|_{L^{p,q}}\\
&=&\|Pf-PQ_{X}f\|_{L^{p,q}}+\|PQ_{X}f-PA_{X,a}f\|_{L^{p,q}}\\
&\leq&\|P\|_{op}(\|f-Q_{X}f\|_{L^{p,q}}+\|Q_{X}f-A_{X,a}f\|_{L^{p,q}})\\
&=&\|P\|_{op}(\|f-Q_{X}f\|_{L^{p,q}}+\|Q_{X}f-Q_{X}(f*\psi^{*}_{a})\|_{L^{p,q}}).
\end{eqnarray*}

Firstly, we estimate the first term of the last inequality.
By Proposition \ref{pro:stableup2}, Proposition \ref{thm:stableup} and Cauchy inequality, one has that there exists $\gamma_{1}$ such that for any $\gamma\leq\gamma_{1}$
\begin{eqnarray}\label{eq:1}
\|f-Q_{X}f\|_{L^{p,q}}&\leq&\left\|\sum_{i=1}^{r}\sum_{k_{1}\in \Bbb Z}\sum_{k_{2}\in \Bbb Z^{d}}|c_{i}(k_{1},k_{2})|\hbox{osc}_{\gamma}(\phi_{i})(\cdot-k_{1},\cdot-k_{2})\right\|_{L^{p,q}}\nonumber\\
&\leq&\sum_{i=1}^{r}\left\|\sum_{k_{1}\in \Bbb Z}\sum_{k_{2}\in \Bbb Z^{d}}|c_{i}(k_{1},k_{2})|\hbox{osc}_{\gamma}(\phi_{i})(\cdot-k_{1},\cdot-k_{2})\right\|_{L^{p,q}}\nonumber\\
&\leq&\sum_{i=1}^{r}\|c_{i}\|_{\ell^{p,q}}\|\hbox{osc}_{\gamma}(\phi_{i})\|_{W(L^{1,1})}\nonumber\\
&\leq&\max_{1\leq i\leq r}\|\hbox{osc}_{\gamma}(\phi_{i})\|_{W(L^{1,1})}\sum_{i=1}^{r}\|c_{i}\|_{\ell^{p,q}}\nonumber\\
&\leq&\max_{1\leq i\leq r}\|\hbox{osc}_{\gamma}(\phi_{i})\|_{W(L^{1,1})}\sqrt{r}\left(\sum_{i=1}^{r}\|c_{i}\|^{2}_{\ell^{p,q}}\right)^{1/2}\nonumber\\
&\leq&D_{2}\sqrt{r}\max_{1\leq i\leq r}\|\hbox{osc}_{\gamma}(\phi_{i})\|_{W(L^{1,1})} \|f\|_{L^{p,q}}.
\end{eqnarray}

Next, we estimate the second term of the last inequality. Let $\phi^{a}_{i}=\phi_{i}-\phi_{i}*\psi^{*}_{a}\,(i=1,\cdots,r)$. Using $\phi_{i}\in W_{0}(L^{1,1})\,(i=1,\cdots,r)$, $\psi\in L^{1}$ and 
Lemma \ref{convolution relation}, it follows that $\phi^{a}_{i}\in W_{0}(L^{1,1})$. Note that $f-f*\psi^{*}_{a}=\sum_{i=1}^{r}\sum_{k\in \bbZ^{d+1}}c_{i}(k)\phi^{a}_{i}(\cdot-k)$. By Lemma \ref{Qoperator},
there exist constants $C,\gamma_{2}>0$ such that for any $\gamma\leq\gamma_{2}$
\begin{eqnarray*}
\|Q_{X}f-Q_{X}(f*\psi^{*}_{a})\|_{L^{p,q}}&\leq&\left\|Q_{X}\left(\sum_{i=1}^{r}\sum_{k\in \bbZ^{d+1}}c_{i}(k)\phi^{a}_{i}(\cdot-k)\right)\right\|_{L^{p,q}}\\
&=&\left\|\sum_{i=1}^{r}Q_{X}\left(\sum_{k\in \bbZ^{d+1}}c_{i}(k)\phi^{a}_{i}(\cdot-k)\right)\right\|_{L^{p,q}}\\
&=&\sum_{i=1}^{r}\left\|Q_{X}\left(\sum_{k\in \bbZ^{d+1}}c_{i}(k)\phi^{a}_{i}(\cdot-k)\right)\right\|_{L^{p,q}}\\
&\leq&C\sum_{i=1}^{r}\|c_{i}\|_{\ell^{p,q}}\|\phi^{a}_{i}\|_{W(L^{1,1})}.
\end{eqnarray*}
Using Proposition \ref{thm:stableup} and Cauchy inequality
\begin{eqnarray*}
\|Q_{X}f-Q_{X}(f*\psi^{*}_{a})\|_{L^{p,q}}&\leq& C\max_{1\leq i\leq r}\|\phi^{a}_{i}\|_{W(L^{1,1})}\sum_{i=1}^{r}\|c_{i}\|_{\ell^{p,q}}\\
&\leq&C\max_{1\leq i\leq r}\|\phi^{a}_{i}\|_{W(L^{1,1})}\sqrt{r}\left(\sum_{i=1}^{r}\|c_{i}\|^{2}_{\ell^{p,q}}\right)^{1/2}\\
&\leq&D_{2}\sqrt{r}C\max_{1\leq i\leq r}\|\phi^{a}_{i}\|_{W(L^{1,1})} \|f\|_{L^{p,q}}.
\end{eqnarray*}

Assume that $\epsilon>0$ is any positive real number. By Proposition \ref{pro:oscillation}, there exists $\gamma_{3}$ such that 
$D_{2}\sqrt{r}\max_{1\leq i\leq r}\|\hbox{osc}_{\gamma}(\phi_{i})\|_{W(L^{1,1})}\leq \epsilon/2$ for any $\gamma\leq\gamma_{3}$. Using Lemma \ref{lem:average function}, there exists $a_{0}>0$ such that 
for any $a\leq a_{0}$
\[
D_{2}\sqrt{r}C\max_{1\leq i\leq r}\|\phi^{a}_{i}\|_{W(L^{1,1})} \leq \epsilon/2.
\]
Hence, we choose $\gamma_{0}=\min\{\gamma_{1},\gamma_{2},\gamma_{3}\}$, one has
\[
\|f-PA_{X,a}f\|_{L^{p,q}}\leq\epsilon\|P\|_{op}\|f\|_{L^{p,q}} \quad\mbox{for any} \, f\in V_{p,q}(\Phi),\,\gamma\leq\gamma_{0},\,a\leq a_{0}.
\]
To get a contraction operator, we choose $\epsilon\|P\|_{op}<1$.
\end{proof}

\section{Proofs of main results}
\ \ \ \
In this section, we give proofs of Theorem \ref{th:suanfa} and Theorem \ref{th:suanfa-m}.

\subsection{Proof of Theorem \ref{th:suanfa}}
\ \ \ \
For convenience, let $e_{n}=f-f_{n}$ be the error after $n$ iterations. Using (\ref{eq:iterative algorithm}),
\begin{eqnarray*}
e_{n+1}&=&f-f_{n+1}\\
&=&f-f_{n}-PA_{X,a}(f-f_{n})\\
&=&(I-PA_{X,a})e_{n}.
\end{eqnarray*}
Using Lemma \ref{lem:co}, we may choose right $\gamma_{0}$ and $a_{0}$ such that for any $\gamma\leq\gamma_{0}$ and $a\leq a_{0}$ 
\[
\|I-PA_{X,a}\|_{op}=\alpha<1. 
\]
Then we obtain
\[
\|e_{n+1}\|_{L^{p,q}}\leq \alpha\|e_{n}\|_{L^{p,q}}\leq \alpha^{n}\|e_{1}\|_{L^{p,q}}.
\]
Wherewith $\|e_{n}\|_{L^{p,q}}\rightarrow0$, when $n\rightarrow\infty$. This completes the proof.

\subsection{Proof of Theorem \ref{th:suanfa-m}}
\ \ \ \
We only need to prove $I-PA_{X}$ is  a contraction operator on $V_{p,q}(\Phi)$.

Let $f=\sum_{i=1}^{r}\sum_{k_{1}\in \Bbb Z}\sum_{k_{2}\in \Bbb Z^{d}}c_{i}(k_{1},k_{2})\phi_{i}(\cdot-k_{1},\cdot-k_{2})\in V_{p,q}(\Phi)$. One has
\begin{eqnarray}\label{eq:2}
\|f-PA_{X}f\|_{L^{p,q}}&=&\|f-PQ_{X}f+PQ_{X}f-PA_{X}f\|_{L^{p,q}}\nonumber\\
&\leq&\|f-PQ_{X}f\|_{L^{p,q}}+\|PQ_{X}f-PA_{X}f\|_{L^{p,q}}\nonumber\\
&=&\|Pf-PQ_{X}f\|_{L^{p,q}}+\|PQ_{X}f-PA_{X}f\|_{L^{p,q}}\nonumber\\
&\leq&\|P\|_{op}(\|f-Q_{X}f\|_{L^{p,q}}+\|Q_{X}f-A_{X}f\|_{L^{p,q}}).
\end{eqnarray}

Putting $f_{i}=\sum_{k_{1}\in \Bbb Z}\sum_{k_{2}\in \Bbb Z^{d}}c_{i}(k_{1},k_{2})\phi_{i}(\cdot-k_{1},\cdot-k_{2}),\,(1\leq i\leq r)$. Then $f_{i}\in V^{p,q}(\phi_{i})$ for $1\leq i\leq r$ and $f=\sum_{i=1}^{r}f_{i}$. For each $f_{i}$, one has
\begin{eqnarray*}
|Q_{X}f_{i}-A_{X}f_{i}|&=&\left|\sum_{j\in\mathbb{J}}\sum_{k\in\mathbb{J}}\left(f_{i}(x_{j},y_{k})-\left \langle f_{i},\psi_{x_{j},y_{k}} \right \rangle\right)\beta_{j,k}\right|\\
&=&\left|\sum_{j\in\mathbb{J}}\sum_{k\in\mathbb{J}}\left(\int_{\bbR^{d+1}}(f_{i}(x_{j},y_{k})-f_{i}(t))\overline{\psi_{x_{j},y_{k}}(t)}dt\right)\beta_{j,k}\right|\\
&\leq&\sum_{j\in\mathbb{J}}\sum_{k\in\mathbb{J}}\int_{\bbR^{d+1}}|f_{i}(x_{j},y_{k})-f_{i}(t)||\psi_{x_{j},y_{k}}(t)|dt\beta_{j,k}\\
&\leq&\sum_{j\in\mathbb{J}}\sum_{k\in\mathbb{J}}\hbox{osc}_{a}(f_{i})(x_{j},y_{k})\int_{\bbR^{d+1}}|\psi_{x_{j},y_{k}}(t)|dt\beta_{j,k}\\
&\leq&M\sum_{j\in\mathbb{J}}\sum_{k\in\mathbb{J}}\hbox{osc}_{a}(f_{i})(x_{j},y_{k})\beta_{j,k}\\
&\leq&M\sum_{j\in\mathbb{J}}\sum_{k\in\mathbb{J}}\sum_{k_{1}\in \Bbb Z}\sum_{k_{2}\in \Bbb Z^{d}}|c_{i}(k_{1},k_{2})|\hbox{osc}_{a}(\phi_{i})(x_{j}-k_{1},y_{k}-k_{2})\beta_{j,k}.
\end{eqnarray*}
By Proposition \ref{pro:stableup2} and Proposition \ref{pro:in Wiener space}, there exists $a_{1}$ such that for any $a\leq a_{1}$
\begin{eqnarray*}
\|Q_{X}f_{i}-A_{X}f_{i}\|_{L^{p,q}}\leq MC\|c_{i}\|_{\ell^{p,q}}\|\hbox{osc}_{a}(\phi_{i})\|_{W(L^{1,1})}.
\end{eqnarray*}
Therewith
\[
\|Q_{X}f-A_{X}f\|_{L^{p,q}}\leq MC\sum_{i=1}^{r}\|c_{i}\|_{\ell^{p,q}}\|\hbox{osc}_{a}(\phi_{i})\|_{W(L^{1,1})}.
\]
Thus by the proof of Lemma \ref{lem:co}
\begin{eqnarray}\label{eq:3}
\|Q_{X}f-A_{X}f\|_{L^{p,q}}\leq MCD_{2}\sqrt{r}\max_{1\leq i\leq r}\|\hbox{osc}_{a}(\phi_{i})\|_{W(L^{1,1})} \|f\|_{L^{p,q}}.
\end{eqnarray}
Using (\ref{eq:1}), (\ref{eq:2}) and (\ref{eq:3}), one has that there exists $\gamma_{1}$ such that for any $\gamma\leq\gamma_{1}$ and $a\leq a_{1}$
\begin{eqnarray*}
\|f-PA_{X}f\|_{L^{p,q}}&\leq&\|P\|_{op}(D_{2}\sqrt{r}\max_{1\leq i\leq r}\|\hbox{osc}_{\gamma}(\phi_{i})\|_{W(L^{1,1})}\\
&&\quad+MCD_{2}\sqrt{r}\max_{1\leq i\leq r}\|\hbox{osc}_{a}(\phi_{i})\|_{W(L^{1,1})} )\|f\|_{L^{p,q}}.
\end{eqnarray*}

Assume that $\epsilon>0$ is any positive real number. By Proposition \ref{pro:oscillation}, there exists $\gamma_{2}$ such that for any $\gamma\leq\gamma_{2}$
$$D_{2}\sqrt{r}\max_{1\leq i\leq r}\|\hbox{osc}_{\gamma}(\phi_{i})\|_{W(L^{1,1})}\leq \epsilon/2.$$
Using Lemma \ref{lem:average function}, there exists $a_{2}>0$ such that 
for any $a\leq a_{2}$
\[
MCD_{2}\sqrt{r}\max_{1\leq i\leq r}\|\hbox{osc}_{a}(\phi_{i})\|_{W(L^{1,1})} \leq \epsilon/2.
\]
Hence, we choose $\gamma_{0}=\min\{\gamma_{1},\gamma_{2}\}$ and $a_{0}=\min\{a_{1},a_{2}\}$, one has
\[
\|f-PA_{X}f\|_{L^{p,q}}\leq\epsilon\|P\|_{op}\|f\|_{L^{p,q}} \quad\mbox{for any} \, f\in V_{p,q}(\Phi),\,\gamma\leq \gamma_{0},\,a\leq a_{0}.
\]
To get a contraction operator, we choose $\epsilon\|P\|_{op}<1$.

\section{Conclusion}
\ \ \ \
In this paper, we study nonuniform average sampling problem in multiply generated shift-invariant subspaces of mixed Lebesgue spaces. We discuss two types of average sampled values. Two fast reconstruction algorithms for this two types of average sampled values are provided. Studying $L^{p,q}$-frames in multiply generated shift-invariant subspaces of mixed Lebesgue spaces is the goal of future work.
\vspace{10pt}
\\
\textbf{Acknowledgements}

This work was supported partially by the
National Natural Science Foundation of China (11326094 and 11401435).



\end{document}